\documentclass{amsart}
\usepackage[paper=a4paper]{geometry}
\usepackage{amssymb}
\usepackage{amsmath,amsfonts,amsthm}
\usepackage{tikz}
\usepackage{graphicx}
\usepackage{capt-of}
\usepackage{lipsum}
\newtheorem{theorem}{Theorem}[section]
\newtheorem{lemma}[theorem]{Lemma}
\newtheorem{corollary}[theorem]{Corollary}
\newtheorem{proposition}[theorem]{Proposition}

\theoremstyle{definition}
\newtheorem{definition}[]{Definition}
\newtheorem{example}[theorem]{Example}

\theoremstyle{remark}
\newtheorem{remark}[theorem]{Remark}

\numberwithin{equation}{section}

\begin{document}
    \title{On gradient $\rho$-Einstein solitons with Bach tensor radially nonnegative}
 
           \author[M. Andrade]{Maria Andrade}
\address[M. Andrade]{Departamento de Matemática, Universidade Federal de Sergipe, S\~ao Cristov\~ao-SE, Brazil
	\newline\indent 
}
%%%\email{\href{mailto: maria@mat.ufs.br}{ maria@mat.ufs.br}}
\email{{ maria@mat.ufs.br}}
       
       \author[V. Borges]{Valter Borges}
\address[V. Borges]{
	Faculdade de Matemática, Universidade Federal do Pará, Bel\'em-PA, Brazil
	\newline\indent 
}

%%%\email{\href{mailto: valterborges@ufpa.br}{ valterborges@ufpa.br}}
\email{{ valterborges@ufpa.br}}

 \author[H. Reis]{Hiuri Reis}
\address[H. Reis]{
	Instituto de Matemática e Estatística, Universidade Federal de Goiás, Goi\^ania-GO, Brazil
	\newline\indent 
}

%%%\email{\href{mailto: hiuri_reis@ufg.br}{hiuri\textunderscore reis@ufg.br}}
\email{{hiuri\textunderscore reis@ufg.br}}

	
	\date{\today}
	
\begin{abstract}
In this paper, we study $n$-dimensional gradient $\rho$-Einstein solitons whose Bach tensor is radially nonnegative. Under this assumption, we show that such $\rho$-Einstein solitons are locally warped products of an interval and an Einstein manifold, provided either $\rho\neq0$ or $\rho=0$ and the soliton is rectifiable. We obtain as a consequence that these solitons must have harmonic Weyl tensor and vanishing Bach tensor. We also finish the classification of complete locally conformally flat steady $\rho$-Einstein solitons and classify these manifolds when their Bach tensor is radially nonnegative and $n\in\{3,4\}$.

\end{abstract}

\maketitle
\section{Introduction and main results}
In 1982, Hamilton \cite{hamilton0} introduced the Ricci flow $\partial_tg=-2Ric,$ where $Ric$ is the Ricci curvature of a Riemannian manifold $(M^n,g)\ n\geq 3$, which has been studied intensively in recent years and plays a key role in Perelman's proof of the Poincaré Conjecture. An essential step in implementing the Ricci flow is the study of Ricci solitons, which generate self-similar solutions to the flow and often arise as singularity models. In particular, gradient Ricci solitons are Riemannian manifolds satisfying $Ric+\nabla^2f=\lambda g,$ where $f\in C^{\infty}(M),$ $\nabla^2f$ is the Hessian of $f$ and $\lambda\in\mathbb{R}$ is a constant.
We can consider on a Riemannian manifold $(M^n,g), n\geq 3,$ geometric flows of the type 
$$\partial_tg=-2(Ric-\rho Rg),$$
for some $\rho \in\mathbb{R}$, which were first studied by Bourguignon and are called Ricci-Bourguignon flows. Motivated by these flows, Catino and Mazzieri \cite{catino} introduced the gradient $\rho$-Eisntein solitons. In this article, we are concerned with the following definition.
%Thus, we bring the following definition.
\begin{definition}
Let $(M^n,g),$ $n\geq 3,$ be a Riemannian manifold and let $\rho\in\mathbb{R}.$ We say that $(M^n,g)$ is a gradient $\rho$-Einstein soliton if there exists a smooth function $f:M^n\to\mathbb{R}$ such that the metric $g$ satisfies the equation
\begin{eqnarray}\label{fundeq}
Ric+\nabla^2f=\left(\rho R+\lambda\right)g,
\end{eqnarray}
for some constant $\lambda\in\mathbb{R},$ where $f$ is called its potential function.
\end{definition}
The function $f$ is called a \textit{potential function} of the $\rho$-Eisntein soliton. The soliton is called {\it shrinking}, {\it steady} or {\it expanding}, provided $\lambda$ is positive, zero or negative, respectively. In this case we use the notation $(M^n,g,f,\lambda)$.
Some special examples are: for $\rho=1/2$ we have the gradient Einstein soliton, for $\rho=1/n$ gradient traceless Ricci soliton and for $\rho=1/2(n-1)$ gradient Schouten soliton.
Different values of $\rho$ give rise to quite different Riemannian manifolds. For instance, if $\rho\neq0$ then the potential function is rectifiable, which means that $\vert\nabla f\vert$ is constant along every regular connected component of the level sets of $f$. In this case we simply say the corresponding $\rho$-Einstein soliton is rectifiable. For more details see \cite{catino,petersen2}. Also, when $\rho\in\{1/2,1/2(n-1),1/n\}$ and the corresponding $\rho$-Einstein soliton is compact, then its potential function is constant. Another difference is that, while it is known that for $\rho\notin\{1/2(n-1),1/n\}$ the corresponding solitons are analytic, nothing is known in the remaining cases. For the proofs see \cite{catino1,ivey}. Interesting phenomena also occurs if $\rho=1/2(n-1)$. It was proved in \cite{catino} that the complete gradient Schouten solitons are Ricci flat if $\lambda=0$ and in \cite{borges} that $0\leq\lambda R\leq2\lambda^2$ if $\lambda\neq0$. Moreover, we recall that $\rho$-Einstein solitons characterize self-similar solutions of certain geometric evolution equations \cite{catino1}.

The simplest examples of $\rho$-Einstein solitons are the Einstein manifolds. Other simple examples are the Gaussian solitons and the rigid solitons. Consult Subsection \ref{subsecExamp} below for these and other examples, and also for the definition of rigidity.

In this article, we investigate the Bach tensor $B$ of a gradient $\rho$-Einstein soliton of dimension $n\geq3$. See Subsection \ref{weyl} for the definition of $B$. This symmetric and trace free $2$-tensor is named after R. Bach \cite{bach}, who introduced it in 1921 to study conformally invariant gravitational theory \cite{fielder}, which is modeled on a four manifold. When $n=4$, the Bach tensor is divergence free, conformally invariant, and its vanishment is a necessary condition for a four manifold to be conformal to an Einstein space \cite[page 18]{berg}. The latter condition turns out to be an equivalence if the manifold is also Kahler \cite[Remark 4]{derd}, or when it is a compact Einstein-Weyl manifold \cite[Corollary 3.2]{swann}. It is also known that compact self-dual Riemannian manifolds have zero Bach tensor \cite[Section 4]{derd}. On top of that, being locally conformally flat or Einstein implies the vanishment of the Bach tensor for any dimension $n\geq3$.
 
Complete gradient Ricci solitons with zero Bach tensor were investigated recently. When $\lambda>0$ and $n\geq4$, Cao and Chen showed in \cite{caoqian} that the vanishment of $B$ implies the harmonicity of the Weyl tensor, allowing them to use the results of \cite{ferlo_0} and \cite{muse} to classify these solitons. This classification says the solitons with $B$ vanishing are quotients of either the Gaussian soliton or the cylinder $\mathbb{R}\times\mathbb{S}^{n-1}$. When $\lambda\geq0$ and assuming curvature bounds, Cao, Catino, Chen, Mantegazza and Mazzieri showed in \cite{caocatino} that the vanishing of $B$ implies rotational symmetry. Thus, if $\lambda=0$, the soliton is homothetic to the Bryant soliton or isometric to the Gaussian soliton. When $\lambda>0$, however, there are examples in \cite[Section 5]{caocatino} showing a more refined classification is not possible.

The strategy used in the classification of \cite{caoqian,caocatino} was to consider a tensor $D$, introduced by Cao and Chen in \cite{caoqian_0}, whose norm is tied to the geometry of the regular level sets of the potential function, and whose vanishment is a powerful tool to analyze the geometry of Ricci solitons in several settings. The proofs in \cite{caoqian,caocatino} use the key integral identity
\begin{align}\label{identity}
    2\int_{M}B(\nabla f,\nabla f)dV=-\int_{M}\vert D\vert^2dV,
\end{align}
where $dV$ is the Riemannian measure $dM$ if $\lambda>0$, and given by $e^{-f}dM$, otherwise. As a consequence, they prove that $B=0$ implies $D=0$. Notice that \eqref{identity} shows actually that $D$ vanishes if $B$ is radially nonnegative, that is, if $B(\nabla f,\nabla f)\geq0$ on $M$.

Bach-flat complete gradient non-steady Schouten solitons were classified in \cite{borges2} by the second author, using a strategy quite different than that used in \cite{caoqian,caocatino}. The key point in this classification is to notice that $\nabla f$ is an eigenvector of $B$ at the regular points of $f$, with eigenvalue
\begin{align}\label{identschb=0}
	\mu=R^2-(n-1)\vert Ric\vert^2.
\end{align}
Thus, the vanishment of the Bach tensor implies $\mu=0$, and the latter condition was shown to be equivalent to $b(s)=\vert\nabla f(\alpha(s))\vert^2$ satisfying an ODE, where $\alpha$ is a suitable curve. This ODE is used to prove that the soliton has constant scalar curvature, allowing to use the results of \cite{ferlo} to obtain the classification. It is worthy mentioning that once $\nabla f$ is in the kernel of $Ric$ (c.f. \eqref{Riccidirec}), we must have $\mu\leq0$. Thus, the results of \cite{borges2} are true assuming $B$ radially nonnegative.

In this paper, we investigate gradient $\rho$-Einstein solitons with Bach tensor radially nonnegative, that is, with $B(\nabla f,\nabla f)\geq0$ on $M$. Our main result, stated below, describes the local geometry of these manifolds.

\begin{theorem}[Main Theorem]\label{main_thm}
	Let $(M^n,g,f,\lambda)$ be a nontrivial gradient $\rho$-Einstein soliton, and $\mathcal{R}\subset M$ be the set of regular values of $f$. If $\rho=0$, assume in addition that the soliton is rectifiable. If $B(\nabla f,\nabla f)\geq0$ on $M$, then, for any $p\in\mathcal{R}$ there is an open connected neighborhood $U\subset\mathcal{R}$ of $p$, an Einstein manifold $N^{n-1}$ and a positive smooth function $h:(-\varepsilon,\varepsilon)\to\mathbb{R}$ so that $(U,\left.g\right\vert_{U})$ is isometric to $(-\varepsilon,\varepsilon)\times_{h}N^{n-1}$.
\end{theorem}

It is important to remark that our strategy differs from the one used in \cite{caoqian,caocatino}. Namely, to prove the result above we do not need to know whether $D$ or $divW$ vanishes identically, circumstances in which such local decompositions would occur. On the contrary, as an immediate consequence of our local decomposition we obtain the following information on the Weyl and Bach tensors.

\begin{corollary}\label{main_cor}
	Let $(M^n,g,f,\lambda)$ be a nontrivial gradient $\rho$-Einstein soliton. If $\rho=0$, assume in addition that the soliton is rectifiable. If $B(\nabla f,\nabla f)\geq0$ on $M$, then the following assertions are true:
    \begin{enumerate}
        \item\label{item1} If $n\geq3$, then $M^{n}$ is Bach flat.
        \item\label{item2} If $n\geq3$, then $M^{n}$ has harmonic Weyl tensor.
        \item\label{item3} If $n\in\{3,4\}$, then $M^{n}$ is locally conformally flat.
    \end{enumerate}
\end{corollary}

Notice that we neither assume the soliton is complete nor impose any curvature bounds in Theorem \ref{main_thm}. As a consequence, we are not in position to show an integral identity similar to \eqref{identity}, and rather adopt a local approach. We first find a nice expression for the Bach tensor in the radial direction, which allows us to conclude that $\nabla f$ is an eigenvector of $B$ at the regular points of $f$, and then compute the corresponding eigenvalue $\mu$, extending \eqref{identschb=0} (see Proposition \ref{Propgradeigv} below). As a matter of fact, the proof of \eqref{identschb=0} in \cite{borges2} is based on the fact that $\nabla f$ is in the kernel of $Ric$, which is not true for $\rho\neq1/2(n-1)$ in general. This difficulty is overcame by using the rectifiability of the soliton. We then concentrate on deciphering the condition $\mu\geq0$. We are not able to proceed with an ODE argument, as done in \cite{borges2}, due to technical reasons. However, we are able to extract information from $\mu\geq0$, namely, that there are at most $2$ distinct eigenvalues of $Ric$, and that $M$ is locally a warped product between an interval and an Einstein manifold. See Theorem \ref{proptwoeigen} and Proposition \ref{local_warp} for these results.

We highlight that the classification of locally conformally flat gradient $\rho$-Einstein solitons is still missing in the literature for $\rho\neq0$, with one exception that we now describe. Catino and Mazzieri constructed in \cite{catino} complete rotationally symmetric steady $\rho$-Einstein soliton for $\rho\in(-\infty,1/2(n-1))\cup[1/(n-1),\infty)$, extending Bryant's construction \cite{bryant} in the case where $\rho=0$, and showed that this is the maximal range for $\rho$ in which such solutions can be found. Assuming $\rho\in(-\infty,0)\cup[1/2,\infty)$, they managed to prove that these are the only complete $n$-dimensional locally conformally flat gradient steady $\rho$-Einstein solitons with positive sectional curvature, up to homotheties. In the next result, as an application of our ideas, we optimize this classification by allowing $\rho$ to be in the maximal range (see item \eqref{thegenerali} below).

\begin{theorem}\label{thm_classif}
	Let $(M^n,g,f,0)$, $n\geq3$, be a complete simply connected and locally conformally flat nontrivial gradient steady $\rho$-Einstein soliton. Then the following happens:
    \begin{enumerate}
        \item If $Ric(\nabla f,\nabla f)$ vanishes identically, then $M$ is isometric to $\mathbb{R}\times\mathbb{H}^{n-1}$, $\mathbb{R}^{n}$ or $\mathbb{R}\times\mathbb{S}^{n-1}$.
        \item\label{thegenerali} If $Ric(\nabla f,\nabla f)>0$ somewhere and the sectional curvature is nonnegative, then $\rho\in(-\infty,1/2(n-1))\cup[1/(n-1),\infty)$ and the soliton is isometric to the type Bryant $\rho$-Einstein soliton constructed by Catino et. al \cite{catino}.
    \end{enumerate}
\end{theorem}

The hypothesis on curvatures in Theorem\ref{thm_classif} cannot be removed. Indeed, Agila and Gomes, in \cite{agila}, constructed examples of non-rigid, complete, three-dimensional steady gradient $(1/3)$-Einstein solitons with negative sectional curvature. These solitons are locally conformally flat but not necessarily rotationally symmetric.

Another application of our results is the following classification in low dimension.

\begin{corollary}\label{cor_W=0}
	Let $(M^n,g,f,0)$ be a complete simply connected and nontrivial gradient steady $\rho$-Einstein soliton with $B(\nabla f,\nabla f)\geq0$ on $M$, and $n\in\{3,4\}$. Then $M$ is locally conformally flat. If in addition $Ric(\nabla f,\nabla f)\geq0$, then it is isometric to one of the solitons given in Theorem \ref{thm_classif}.
\end{corollary}

We would like to point out that identity \eqref{identity} was extended in \cite{patuho} for a complete shrinking $\rho$-Einstein soliton with $\rho\neq1/2(n-1)$, uniformly bounded Ricci curvature and $\rho R\geq0$. As a consequence, it is shown that a Bach flat $\rho$-Einstein soliton satisfying these bounds on the curvature has harmonic Weyl curvature. The approach used is rather inspired by the one developed in \cite{caoqian}, which differs from ours. In particular, our results show that these extra assumptions are not necessary for the conclusion that the soliton has harmonic Weyl tensor when $\rho\neq0$.

In this article we extend identity \eqref{identschb=0} to all values of $\rho\neq0$. When $\rho=0$, we need to assume in addition that the soliton is rectifiable. Even though there are many rectifiable gradient solitons in the literature \cite{bdw,bdw2,dw,dw2}, as observed in \cite{catino}, the product of any two of these solitons is not rectifiable.
\begin{remark}\label{remarkRectf_rho0}
It follows by Hamilton's identity \eqref{hamidenEQ} that the rectifiability of Ricci solitons is equivalent to $df\wedge dR=0$. This last condition is true if we assume $B\equiv0$ and $M$ complete, as shown by the works \cite{caoqian,caocatino}. 
\end{remark}

This paper is organized in the following way: in Section \ref{preliminariesresults} we  fix some notations and review some basic facts about gradient $\rho$-Eisntein soliton. Then, in Section \ref{sectionbachT} we compute the Bach tensor of a $\rho$-Einstein soliton to prove that $\nabla f$ is an eigenvector of this tensor. After, in Section \ref{locstruc} we obtain a local structure of $\rho$-Eisntein soliton under the hypothesis $B(\nabla f,\nabla f)\geq0$. Finally, in the Section \ref{locaconf} we show that complete locally conformally flat $\rho$-Einstein solitons with nonnegative sectional curvature are rotationally symmetric. In the sequence, we classify those for which $\lambda=0$.

\section{Notations, preliminaries and examples}\label{preliminariesresults}

In this section, we review some basic facts, notations, propositions and collect examples that will be useful in this work.

\subsection{Notation}
 Here we consider the following convention for the curvature operator $Rm$, presented in coordinates and considering Einstein's notation for repeated indexes:
\begin{align*}
	Rm(\partial_{i},\partial_{j})\partial_{l}=\nabla_{j}\nabla_{i}\partial_{l}-\nabla_{i}\nabla_{j}\partial_{l}=R_{ijk}^{\ \ \ r}\partial_{r},
\end{align*}
and 
\begin{align*}
	R_{ijkl}=g_{lr}R_{ijk}^{\ \ \ r}.
\end{align*}

With this convention, the Ricci tensor and scalar curvature are given by 
\begin{align*}
	R_{ik}=g^{jl}R_{ijkl}\ \ \ \text{and}\ \ \ R=g^{ik}R_{ik},
\end{align*}
respectively.

Moreover the Riemman's identity reads as
\begin{align}\label{riccident}
	\nabla_{i}\nabla_{j}\nabla_{k}f-\nabla_{j}\nabla_{i}\nabla_{k}f=R_{ijkl}\nabla^{l}f.
\end{align}

We take the opportunity to recall the following identity
\begin{align}\label{divricc}
	(div Rm)_{ijk}=\nabla_{j}R_{ik}-\nabla_{i}R_{jk},
\end{align}
where $div Rm$ is the $3$-tensor defined by
\begin{align}
	(div Rm)_{ijk}=\nabla^{l}R_{ijkl}.
\end{align}

Along the next sections we will also use freely the rules for raising and lowering indexes. According to these rules and considering tensors $A_{ij}$ and $T_{ijkl}$, we have, for example, the equalities $A_{ik}T^{i\ k\ }_{\ j\ l}=A^{i\ }_{\ k}T^{\ \ k}_{ij\ l}=A_{i\ }^{\ k}T^{i}_{\ jkl}$.

\subsection{Preliminary results}

We now collect some well-known identities and results concerning $\rho$-Einstein solitons. 

\begin{proposition}[\cite{catino}]If $(M^n,g,f,\lambda)$ is a gradient $\rho$-Einstein soliton, then
	\begin{align}\label{wedgezero}
		\rho df\wedge dR=0,
	\end{align}
	\begin{align}\label{trace}
		\Delta f=n\lambda-(1-n\rho)R,
	\end{align}
	\begin{align}\label{Riccidirec}
		Ric(\nabla f,X)=\frac{1}{2}(1-2(n-1)\rho)g(\nabla R,X),\ \forall X\in\mathfrak{X}(M),
	\end{align}
	\begin{align}\label{IdentitySch}
		(1-2(n-1)\rho)\Delta R=\left\langle\nabla f,\nabla R\right\rangle-2|Ric|^2+2R\left(\rho R+\lambda\right).
	\end{align}
\end{proposition}

When $\rho\neq0$, identity \eqref{wedgezero} was used in \cite{catino} to prove the following theorem.

\begin{theorem}[Theorem 1.2 of \cite{catino}]
	Every gradient $\rho$-Einstein soliton with $\rho\neq0$ is rectifiable.
\end{theorem}

This result is one of the main differences between the cases $\rho=0$ and $\rho\neq0$. When $\rho=0$, rectifiability follows from the vanishment of certain geometric tensors, such as the Weyl or Bach tensor (see \cite{caoqian_0,caoqian,catino}), however, as observed by Catino and Mazzieri in \cite{catino}, there are examples of Ricci solitons which are not rectifiable. We observe that rectifiable Ricci solitons were investigated in \cite{petersen2}.

Quantitatively, the identity \eqref{wedgezero} has the following consequence, which shall be important in the next sections.

\begin{proposition}\label{eigenvRic_stat}
	Let $(M^n,g,f,\lambda)$ be a gradient $\rho$-Einstein soliton. If $\rho=0$, assume in addition that the soliton is rectifiable. Then $\nabla f$ is an eigenvector of $Ric$, whenever it does not vanish. More precisely:
	\begin{align}\label{eigenvRic}
		Ric(\nabla f,X)=\frac{1}{2}(1-2(n-1)\rho)\frac{\left\langle\nabla R,\nabla f\right\rangle}{\vert\nabla f\vert^2} g(\nabla f,X).
	\end{align}
\end{proposition}

\begin{proof}
	First observe that when $\rho\neq0$, identity \eqref{wedgezero} implies
	\begin{align*}
		\left\langle\nabla R,\nabla f\right\rangle\left\langle\nabla f,X\right\rangle=\vert\nabla f\vert^2\left\langle\nabla R,X\right\rangle,
	\end{align*}
	for any $X\in\mathfrak{X}(M)$. This is equivalent to
	\begin{align}\label{wedgezero2}
		\nabla R=\frac{\left\langle\nabla R,\nabla f\right\rangle}{\vert\nabla f\vert^2}\nabla f.
	\end{align}
	We now use the equality above in \eqref{Riccidirec}, proving \eqref{eigenvRic}. Lastly, notice that when the Ricci soliton is rectifiable, we also have $df\wedge dR=0$, which by using the same argument, implies the validity of \eqref{eigenvRic}.
\end{proof}

On the other hand, Hamilton proved in \cite{hamilton2} that when $\rho=0$ there is a constant $C$ so that
	\begin{align}\label{hamidenEQ}
		\vert\nabla f\vert^2+R-2\lambda f=C.
	\end{align}

Identity \eqref{hamidenEQ} is called Hamilton identity. It is an important tool in the development of the Ricci soliton theory. Unfortunately, it is still not known whether a similar identity holds true for $\rho$-Einstein solitons with $\rho\neq0$.

In studying these metrics, it is important to know properties of the set $\mathcal{R}\subset M$ of regular values of $f$. This issue was addressed in \cite{borges2,borges_gomes,catino1}, for example. Below we recall some of these results, that are going to be used in this work. First, we have analiticity for all but two values of $\rho$.

\begin{theorem}[\cite{catino1}]
	Let $(M^n,g,f,\lambda)$ be a $\rho$-Einstein soliton with $\rho\notin\{1/n,1/2(n-1)\}$. Then, in harmonic coordinates, the metric $g$ and the potential function $f$ are real analytic.
\end{theorem}

When $\rho=1/2(n-1)$, the following result was proved in \cite{borges2}.

\begin{proposition}[\cite{borges2}]\label{dense}
	Let $(M^n,g,f,\lambda)$, $\lambda\neq0$, be a complete Schouten soliton with $f$ nonconstant. The set of regular points of $f$ is dense in $M$.
\end{proposition}

For the remaining value of $\rho$, Gomes and the second author proved that:

\begin{proposition}[\cite{borges_gomes}]\label{denseEinstein}
	Let $(M^n,g,f,\lambda)$ be a traceless soliton with $f$ nonconstant. The set of regular points of $f$ is dense in $M$.
\end{proposition}

For steady Ricci solitons, the situation is resolved below.

\begin{theorem}[\cite{catino}]
    Let $(M^n,g,f,0)$ be a complete gradient steady Schouten soliton. Then $M$ is Ricci flat.
\end{theorem}

Thus, nontrivial complete noncompact gradient steady Schouten solitons are isometric to a product of a steady Gaussian soliton and a Ricci flat manifold.

As a consequence of these results, we have:

\begin{corollary}[Density of $\mathcal{R}$]\label{densedense}
	Let $(M^n,g,f,\lambda)$ be a complete $\rho$-Einstein soliton with $f$ nonconstant. Then $\mathcal{R}$ is dense in $M$.
\end{corollary}

\subsection{Examples of $\rho$-Einstein soliton}\label{subsecExamp}

In this section, we recall examples of $\rho$-Einstein solitons, besides the Einstein ones, for which the potential function is constant. The next two examples have constant scalar curvature, and appear in many classification results (e.g., \cite{catino,catino1,borges2}).

\begin{example}[Gaussian soliton]
	A Gaussian $\rho$-Einstein soliton is the Euclidean space endowed its canonical metric, and whose potential function is given by
	\begin{align*}
			f(x)=\frac{1}{2}\lambda\vert x\vert^2+\left\langle v,x\right\rangle,
	\end{align*}
	for fixed $\lambda\in\mathbb{R}$ and $v\in\mathbb{R}^n$.
\end{example}

\begin{example}[Rigid solitons]\label{example}
	Given an Einstein manifold $(N^{k},g_{N})$ with scalar curvature $R_{N}$, consider $\rho,\ \lambda\in\mathbb{R}$ so that $(1-k\rho)R_{N}=\lambda k$ and a Gaussian soliton $\mathbb{R}^{m}$, with potential function
		$$f(x)=\frac{1}{2}(\rho R_{N}+\lambda)\vert x\vert^2+\left\langle v,x\right\rangle,$$
	with $v\in\mathbb{R}^n$ fixed. Then, $(\mathbb{R}^{m}\times N^{k},g_{can}+g_{N},f,\lambda)$ is a $\rho$-Einstein soliton, with $f(x,p)=f(x)$.
\end{example}

A Riemannian manifold is called {\it rigid} if it is isometric to one of those described in Example \ref{example}. An important reference on rigid Ricci solitons is \cite {petersen}.

In \cite{bryant}, Robert Bryant constructed an example of complete and rotationally symmetric steady Ricci soliton. Below, we state Catino and Mazzieri's theorem \cite{catino}, which extends the example of Bryant to $\rho$-Einstein solitons with $\rho\notin[1/2(n-1),1/(n-1))$. Following \cite{catino}, a complete simply connected gradient $\rho$-Einstein soliton $(M,g,f,\lambda)$ is a warped product with canonical fiber, if
\begin{align}\label{canon_fiber}
    g=dt^2+h(t)^2g_{can},\ \text{in}\ M^2\backslash\Lambda,
\end{align}
where $g_{can}$ is a metric of constant curvature on a $(n-1)$-dimensional manifold, $t\in(t^{*},t_{*})$, $-\infty\leq t_{*}<t^{*}\leq+\infty$, $h:(t^{*},t_{*})\to(0,+\infty)$ is a smooth function and $\Lambda\subset M$ is a set with at most two points.

They proved the following result:
\begin{theorem}[\cite{catino}]\label{rotat_=0}
    If $\rho<1/2(n-1)$ or $\rho\geq1/(n-1)$, $n\geq3$, then, up to homotheties, there exists a unique complete non-compact, gradient steady $\rho$-Einstein soliton which is a warped product with canonical fibers as in \eqref{canon_fiber} with $Ric(\partial_{t},\partial_{t})(t_0)>0$ at some $t_{0}\in(t^{*},t_{*})$. This solution is rotationally symmetric and has positive sectional curvature.

    If $1/2(n-1)\leq\rho<1/(n-1))$, then there are no complete, non-compact, gradient steady $\rho$-Einstein solitons which are warped products with canonical fibers as in \eqref{canon_fiber} with $Ric(\partial_{t},\partial_{t})(t_0)
    >0$ for some $t_{0}\in(t^{*},t_{*})$.
\end{theorem}

For geometric properties of these solitons see Proposition 4.6 in \cite{catino}. The following uniqueness result was also proved by Catino and Mazzieri in the same article.

\begin{theorem}[\cite{catino}]
    Up to homotheties, there is only one complete $n$-dimensional locally conformally flat gradient steady $\rho$-Einstein soliton with $\rho<0$ or $\rho\geq1/2$ and positive sectional curvature, namely, the rotationally symmetric one, constructed in Theorem \ref{rotat_=0}.
\end{theorem}

When $n=3$, the uniqueness is stronger, as shown below.

\begin{theorem}[\cite{catino}]\label{rotat_=03}
    Up to homotheties, there is only one complete $3$-dimensional gradient steady $\rho$-Einstein soliton with $\rho<0$ or $\rho\geq1/2$ and positive sectional curvature, namely, the rotationally symmetric one, constructed in Theorem \ref{rotat_=0}.
\end{theorem}

If one requires that $Ric(\partial_{t},\partial_{t})=0$, then the following non-existence was established.

\begin{theorem}[\cite{catino}]\label{rotat_=04}
    Let $(M^n,g,f,\lambda)$, $n\geq3$, be a complete, noncompact gradient $\rho$-Einstein soliton which is a warped product with canonical fibers as in \eqref{canon_fiber}. If $Ric(\partial_{t},\partial_{t})=0$, then $(M,g)$ is either homothetic to the round cylinder $\mathbb{R}\times\mathbb{S}^{n-1}$, or to the hyperbolic cylinder $\mathbb{R}\times\mathbb{H}^{n-1}$, or to the flat $\mathbb{R}^{n}$.
\end{theorem}

One may wonder what happens if $Ric(\partial_{t},\partial_{t})<0$ at some $t_{0}\in(t^{*},t_{*})$. We are aware of two explicit examples of three dimensional complete locally conformally flat traceless Ricci solitons, due to the work of Agila and Gomes \cite{agila}. We present them below adopting the point of view of \cite{borges_gomes}, by Gomes and the second author.

\begin{example}[Example 1 in \cite{agila}]\label{exag_1}
    Consider the warped product $\mathbb{R}\times_{h}F^{2}$ and the function $f:\mathbb{R}\times F^{2}\rightarrow\mathbb{R}$, where $Ric_{F}=0$ and
    \begin{align*}
        &h(t)=\displaystyle\sqrt{\frac{e^{2t}+1}{e^{2t}}},\\
        &f(t,x)=\frac{1}{3}\log(e^{2t}+1),
    \end{align*}
    $(t,x)\in\mathbb{R}\times F^{2}$. A straightforward computation shows that the equation
    \begin{align*}
        Ric+\nabla^2 f=\frac{1}{3}Rg,
    \end{align*}
    is satisfied, making $(\mathbb{R}\times_{h}F^{2},g,f,0)$ a complete steady $(1/3)$-Einstein soliton, that is, a traceless Ricci soliton. 
\end{example}

\begin{example}[Example 2 in \cite{agila}]\label{exag_2}
       Consider the warped product $\mathbb{R}\times_{h}F^{2}$ and the function $f:\mathbb{R}\times F^{2}\rightarrow\mathbb{R}$, where $Ric_{F}=0$ and
    \begin{align*}
        &h(t)=\sqrt{t^2+1}\\
        &f(t,x)=\frac{1}{12}(4\log(t^2+1)+2t^2+1),
    \end{align*}
    $(t,x)\in\mathbb{R}\times F^{2}$. A straightforward computation shows that the equation
    \begin{align*}
        Ric+\nabla^2 f=\left(\frac{1}{3}R+\frac{1}{3}\right)g,
    \end{align*}
    is satisfied, making $(\mathbb{R}\times_{h}F^{2},g,f,1/3)$ a complete shrinking $(1/3)$-Einstein soliton, that is, a traceless Ricci soliton.
\end{example}

%\begin{remark}
  We finish this section observing that the traceless Ricci solitons of examples \ref{exag_1} and \ref{exag_2} are locally conformally flat. However, as the fiber is Ricci flat, the soliton is not rotationally symmetric (see item $(4)$ of Theorem 3 in \cite{choi}).
%\end{remark}

\section{The Bach Tensor of a $\rho$-Einstein Soliton}\label{sectionbachT}

In this section, we compute the Bach tensor of a $\rho$-Einstein soliton. In Subsection \ref{weyl}, we recall the definitions of the tensors that we will need, specially the definition of the Bach tensor. In Subsection \ref{eigenvalue_Bach}, we use the rectifiability of the gradient $\rho$-Einstein soliton to prove that $\nabla f$ is an eigenvector of $B$. See Proposition \ref{Propgradeigv} below. In the next section, we investigate the algebraic and geometric content of the vanishment of the eigenvalue corresponding to $\nabla f$, which afterward leads to the proof of our main theorem. 

\subsection{The Weyl and Cotton tensors}\label{weyl}
Let us recall some definitions. Let $(M^n,g)$ be a Riemannian manifold. Here we use the notations of \cite{chowluni}. For any $n\geq3$ its {\it Weyl} and {\it Cotton} tensors are defined, respectively, by
\begin{align}
	W_{ijkl}=&R_{ijkl}-\frac{1}{n-2}(g_{ik}R_{jl}-g_{il}R_{jk}-g_{jk}R_{il}+g_{jl}R_{ik})\label{weyl}\\
	&+\frac{R}{(n-1)(n-2)}(g_{ik}g_{jl}-g_{il}g_{jk}),\nonumber\\
	C_{ijk}=&\nabla_{i}R_{jk}-\nabla_{j}R_{ik}-\frac{1}{2(n-1)}(g_{jk}\nabla_{i}R-g_{ik}\nabla_{j}R)\label{cotton},
\end{align}
and for $n\geq4$ the {\it Bach} tensor is defined by
\begin{align}\label{origbach}
	B_{ij}=&\frac{1}{n-3}\nabla^{k}\nabla^{l}W_{ikjl}+\frac{1}{n-2}R_{kl}W_{i\ j}^{\ k\ l}.
\end{align}
In terms of the Cotton tensor, $B$ can be rewritten by
\begin{align}\label{bhct}
	(n-2)B_{ij}=&\nabla^{k}C_{kij}+R_{kl}W_{i\ j}^{\ k\ l}.
\end{align}
As in \cite{caoqian}, we use the fact that $(\ref{bhct})$ is well defined for $n=3$ to define the Bach tensor in this dimension. More precisely, the Bach tensor of $(M^3,g)$ is defined by
\begin{align}\label{bhctn=3}
	B_{ij}=&\nabla^{k}C_{kij}.
\end{align}
One says that a Riemannian manifold $M^n$ is {\it locally conformally flat} if $n\geq4$ and its Weyl tensor vanishes or, if $n=3$ and its cotton tensor vanishes. $M^n$ is said to be {\it Bach-flat} when its Bach tensor vanishes, $n\geq3$.
Next, we obtain the eigenvalues of the Bach tensor of a $\rho$-Einstein soliton. 

\subsection{Eigenvalues of the Bach tensor of a $\rho$-Einstein soliton}\label{eigenvalue_Bach}
In this subsection we show that $\nabla f$ is an eigenvector of the Bach tensor, at the regular points of $p$. At our first Lemma we rewrite the Bach tensor using the $\rho$-Einstein soliton equation and the definition of the Weyl tensor.

\begin{lemma}\label{lemma1} The Bach tensor of a gradient $\rho$-Einstein soliton $(M^n,g,f,\lambda)$ with $n\geq3$ is given by
		\begin{align}\label{firstbachstep}
			\begin{split}
				(n-2)B=&\frac{2}{n-2}Ric^2+\left[\lambda-\left(\frac{n-(n-1)(n-2)\rho}{(n-1)(n-2)}\right) R\right]Ric\\
				&-\left[\frac{1}{n-2}\left(\vert Ric\vert^2-\frac{R^2}{n-1}\right)+\left(\frac{(1-2(n-1)\rho)}{2(n-1)}\right)\Delta R\right]g\\
				&+\left(\frac{(1-2(n-1)\rho)}{2(n-1)}\right)\nabla^2 R+(div Rm)(\cdot,\nabla f,\cdot),
			\end{split}
		\end{align}
	where $(Ric^2)_{ij}=g^{kl}R_{ik}R_{lj}=R_{i\ }^{\ l}R_{lj}$.
\end{lemma}
\begin{proof}
To write the Bach tensor for a $\rho$-Eisntein manifold, we are going to consider its Cotton tensor. Then, from equation \eqref{fundeq} and definition \eqref{cotton}, we get
\begin{align}\label{cotton_rho}
\begin{split}
	C_{ijk}=&\rho \nabla_{i}Rg_{jk}-\nabla_{i}\nabla_{j}\nabla_{k}f-\rho \nabla_{j}Rg_{ik}+\nabla_{j}\nabla_{i}\nabla_{k}f-\frac{1}{2(n-1)}(g_{jk}\nabla_{i}R-g_{ik}\nabla_{j}R)\\
	=&\left(\rho-\frac{1}{2(n-1)}\right)(\nabla_{i}Rg_{jk}-\nabla_{j}Rg_{ik})+\nabla_{j}\nabla_{i}\nabla_{k}f-\nabla_{i}\nabla_{j}\nabla_{k}f\\
	=&-\left(\frac{1-2(n-1)\rho}{2(n-1)}\right)(\nabla_{i}Rg_{jk}-\nabla_{j}Rg_{ik})+R_{jikl}\nabla_{l}f,
\end{split}
\end{align}
where in the last line we have used \eqref{riccident}. Consequently,
\begin{align*}
	\nabla^{i}C_{ijk}=&-\left(\frac{1-2(n-1)\rho}{2(n-1)}\right)(\nabla^{i}\nabla_{i}Rg_{jk}-\nabla^{i}\nabla_{j}Rg_{ik})+\nabla^{i}R_{jikl}\nabla^{l}f+R_{jikl}\nabla^{i}\nabla^{l}f\\	
	=&-\left(\frac{1-2(n-1)\rho}{2(n-1)}\right)(\Delta Rg_{jk}-\nabla_{k}\nabla_{j}R)+\nabla^{i}R_{klji}\nabla^{l}f+R_{jikl}((\rho R+\lambda)g^{il}-R^{il})\\
	=&-\left(\frac{1-2(n-1)\rho}{2(n-1)}\right)(\Delta Rg_{jk}-\nabla_{k}\nabla_{j}R)+(div Rm)_{klj}\nabla^{l}f+(\rho R+\lambda)R_{jk}-R_{jikl}R^{il}.
\end{align*}

Now, from identity \eqref{bhct} when $n\geq4$, we infer
\begin{align}\label{partn1}
	\begin{split}
	(n-2)B_{jk}=&\nabla^{i}C_{ijk}+R_{il}W_{j\ k}^{\ i\ l}\\
	=&-\left(\frac{1-2(n-1)\rho}{2(n-1)}\right)(\Delta Rg_{jk}-\nabla_{k}\nabla_{j}R)+(div Rm)_{klj}\nabla^{l}f\\
	&+(\rho R+\lambda)R_{jk}+R^{il}(W_{jikl}-R_{jikl}).
	\end{split}
\end{align}
For the last term, on the other hand, we obtain from \eqref{weyl} that
\begin{align}\label{part2}
\begin{split}
		R^{il}(W_{jikl}-R_{jikl})=&\frac{1}{n-2}(g_{ik}R^{il}R_{jl}-g_{il}R^{il}R_{jk}-g_{jk}R^{il}R_{il}+g_{jl}R^{il}R_{ik})\\
	&-\frac{R}{(n-1)(n-2)}(g_{ik}g_{jl}R^{il}-g_{il}g_{jk}R^{il})\\
	=&\frac{1}{n-2}(R_{k\ }^{\ l}R_{jl}-RR_{jk}-\vert Ric\vert^2g_{jk}+R^{i\ }_{\ j}R_{ik})\\
	&-\frac{R}{(n-1)(n-2)}(R_{jk}-Rg_{jk})\\
	=&\frac{2}{n-2}(Ric^2)_{jk}-\frac{1}{n-2}\left(1+\frac{1}{n-1}\right)RR_{jk}\\
	&+\frac{1}{n-2}\left(\frac{R^2}{n-1}-\vert Ric\vert^2\right)g_{jk}\\
	=&\frac{2}{n-2}(Ric^2)_{jk}-\frac{n}{(n-1)(n-2)}RR_{jk}\\
	&-\frac{1}{n-2}\left(\vert Ric\vert^2-\frac{R^2}{n-1}\right)g_{jk}.
\end{split}
\end{align}

Inserting \eqref{part2} in \eqref{partn1} and grouping the similar terms, we obtain
\begin{align*}
		(n-2)B_{jk}=&-\left(\frac{1-2(n-1)\rho}{2(n-1)}\right)(\Delta Rg_{jk}-\nabla_{k}\nabla_{j}R)+(div Rm)_{klj}\nabla^{l}f\\
		&+(\rho R+\lambda)R_{jk}+\frac{2}{n-2}(Ric^2)_{jk}-\frac{n}{(n-1)(n-2)}RR_{jk}\\
		&-\frac{1}{n-2}\left(\vert Ric\vert^2-\frac{R^2}{n-1}\right)g_{jk}\\
		=&\frac{2}{n-2}(Ric^2)_{jk}+\left(\rho R+\lambda-\frac{n}{(n-1)(n-2)}R\right)R_{jk}\\
		&-\left(\frac{1}{n-2}\left(\vert Ric\vert^2-\frac{R^2}{n-1}\right)+\left(\frac{1-2(n-1)\rho}{2(n-1)}\right)\Delta R\right)g_{jk}\\
		&+\left(\frac{1-2(n-1)\rho}{2(n-1)}\right)\nabla_{k}\nabla_{j}R+(div Rm)_{klj}\nabla^{l}f,
\end{align*}
which proves \eqref{firstbachstep}.

For the case $n=3$, similar computations using that $W_{jikl}$ vanishes, \eqref{bhctn=3} and \eqref{part2} give the desired result.
\end{proof}

Next, we use \eqref{eigenvRic} to eliminate $\nabla^2R$ and $\Delta R$ from \eqref{firstbachstep}. First we introduce some notation. Recall that $\mathcal{R}\subset M$ is the set of regular points of $f$. If $p\in \mathcal{R}$, then we consider an orthonormal base $\{e_{1},\ldots,e_{n}\}$ of $T_{p}M$ which diagonalizes $Ric$, where we fix
\begin{align}
	e_{1}=\frac{\nabla f(p)}{\vert\nabla f(p)\vert}.
\end{align}
We also write the associated eigenvalues as $\xi_{1},\ldots,\xi_{n}$, that is,
\begin{align}\label{alleig}
	Ric(e_{i},u)=\xi_{i}g(e_{i},u),\ \forall u\in T_{p}M.
\end{align}
With this notation and from \eqref{eigenvRic} we have
\begin{align}\label{xi1}
	\xi_{1}=\frac{1}{2}(1-2(n-1)\rho)\frac{\left\langle\nabla R(p),\nabla f(p)\right\rangle}{\vert\nabla f(p)\vert^2},
\end{align}
that is a smooth function in $\mathcal{R}$.

The goal of the following lemma is to compute $\nabla^2 R$ and $\Delta R$ in terms of $\xi_{1}$.

\begin{lemma}\label{lemmahesslaplR}
	Let $(M^n,g,f,\lambda)$ be a gradient $\rho$-Einstein soliton with $n\geq3$. If $\rho=0$, assume in addition that $dR\wedge df=0$. At any point of $\mathcal{R}$ we have
	\begin{align}
		&\frac{1}{2}(1-2(n-1)\rho)\nabla^2 R=\xi_{1}((\rho R+\lambda)g-Ric)+d\xi_{1}\otimes df\label{firstdereigID}\\
		&\frac{1}{2}(1-2(n-1)\rho)\Delta R=\xi_{1}(n\lambda-(1-n\rho)R)+\left\langle\nabla\xi_{1},\nabla f\right\rangle\label{trfirstdereigID}
	\end{align}
\end{lemma}
\begin{proof}
	First, we notice that \eqref{wedgezero2} can be written as
	\begin{align}\label{wedgezero2'}
		\frac{1}{2}(1-2(n-1)\rho)\nabla R=\xi_{1}\nabla f.
	\end{align}
	Thus, \eqref{firstdereigID} follows from taking the covariant derivative of \eqref{wedgezero2'} and then using \eqref{fundeq}. Tracing \eqref{firstdereigID} we finally get \eqref{trfirstdereigID}.
\end{proof}

As a consequence of the last lemma, we are able to express the Bach tensor of a $\rho$-Einstein soliton only in terms of $div Rm$, $Ric$ and $R$, besides of course of $\xi_{1}$ and its gradient. As we will see ahead, $div Rm(\cdot,\nabla f,\cdot)$ is easily computed in the directions of the eigenvectors of $Ric$, especially in the direction of $\nabla f$, showing the expression below is useful. 

\begin{lemma}
	Let $(M^n,g,f,\lambda)$ be a gradient $\rho$-Einstein soliton with $n\geq3$. If $\rho=0$, assume in addition that $dR\wedge df=0$. Then the following expression holds within $\mathcal{R}$:
	\begin{align}\label{secndbachstep}
		\begin{split}
			(n-2)B=&\frac{2}{n-2}Ric^2+\left[\left(\frac{(n-1)(n-2)\rho-n}{(n-1)(n-2)}\right) R+\lambda-\frac{1}{n-1}\xi_{1}\right]Ric\\
		&+\left[\left(\frac{(1-(n-1)\rho)}{n-1}R-\lambda\right)\xi_{1}-\frac{1}{n-1}\left\langle\nabla\xi_{1},\nabla f\right\rangle-\frac{1}{n-2}\left(\vert Ric\vert^2-\frac{R^2}{n-1}\right)\right]g\\
		&+\frac{1}{n-1}df\otimes d\xi_{1}+(div Rm)(\cdot,\nabla f,\cdot),
		\end{split}
	\end{align}
	where $(Ric^2)_{ij}=g^{kl}R_{ik}R_{lj}=R_{i\ }^{\ l}R_{lj}$.
\end{lemma}
\begin{proof}
	We do not consider in the following computations the first and last terms of the right hand side of \eqref{firstbachstep}, as they will not change. For the rest of the terms denoted by $P$, we use Lemma \ref{lemmahesslaplR} to get:
	\begin{eqnarray*}
		P&:=&\left(\lambda-\left(\frac{n-(n-1)(n-2)\rho}{(n-1)(n-2)}\right) R\right)R_{jk}+\left(\frac{1-2(n-1)\rho}{2(n-1)}\right)\nabla_{k}\nabla_{j}R\\
		&-&\left(\frac{1}{n-2}\left(\vert Ric\vert^2-\frac{R^2}{n-1}\right)+\left(\frac{1-2(n-1)\rho}{2(n-1)}\right)\Delta R\right)g_{jk}\\
		&=&\left(\lambda-\left(\frac{n-(n-1)(n-2)\rho}{(n-1)(n-2)}\right) R\right)R_{jk}\\
		&+&\frac{1}{n-1}\left(\xi_{1}((\rho R+\lambda)g_{kj}-R_{kj})+\nabla_{k}\xi_{1}\nabla_{j}f\right)\\
		&-&\left(\frac{1}{n-2}\left(\vert Ric\vert^2-\frac{R^2}{n-1}\right)+\frac{1}{n-1}(\xi_{1}(n\lambda-(1-n\rho)R)+\left\langle\nabla\xi_{1},\nabla f\right\rangle)\right)g_{jk}\\
		&=&\left(\lambda-\left(\frac{n-(n-1)(n-2)\rho}{(n-1)(n-2)}\right) R-\frac{1}{n-1}\xi_{1}\right)R_{jk}+\frac{1}{n-1}\nabla_{k}\xi_{1}\nabla_{j}f\\
		&+&\frac{1}{n-1}\xi_{1}(\rho R+\lambda)g_{kj}\\
		&-&\left(\frac{1}{n-2}\left(\vert Ric\vert^2-\frac{R^2}{n-1}\right)+\frac{1}{n-1}(\xi_{1}[(n-1)\lambda-(1-(n-1)\rho)R]+\left\langle\nabla\xi_{1},\nabla f\right\rangle)\right)g_{jk}.
	\end{eqnarray*}
	Inserting this in \eqref{firstbachstep} we obtain \eqref{secndbachstep}.
\end{proof}

Having established the lemmas above, we are ready to show that $\nabla f$ is an eigenvector of the Bach tensor within $\mathcal{R}$, and compute the corresponding eigenvalue.

\begin{proposition}\label{Propgradeigv}
	Let $(M^n,g,f,\lambda)$ be a gradient $\rho$-Einstein soliton with $n\geq3$. If $\rho= 0$, assume in addition that $dR\wedge df=0$. Consider the following function, defined at a regular point of $f$
	\begin{align}\label{eigbach}
		\mu=n\xi_{1}^{2}-2R\xi_{1}-((n-1)\vert Ric\vert^2-R^2).
	\end{align}
	At a regular point of $f$, the Bach tensor of $M$ satisfies
\begin{align}\label{gradeigv}
	B(\nabla f,X)=\frac{\mu}{(n-1)(n-2)^2}g(\nabla f,X),\ \ \forall X\in\mathfrak{X}(M).
\end{align}
\end{proposition}
\begin{proof}
	We will use \eqref{secndbachstep} to compute $B_{kj}\nabla^{k}f$. First notice that
	\begin{align*}
		(div Rm)_{klj}\nabla^{l}f\nabla^{k}f=(div Rm)(\nabla f,\nabla f,\partial_{j})=0,
	\end{align*}
	due to the Second Bianchi identity. Thus, using that $R_{kj}\nabla^{k}f=\xi_{1}\nabla_{j}f$, we have
	\begin{align*}
			(n-2)B_{kj}\nabla^{k}f=&\frac{2}{n-2}(Ric^2)_{kj}\nabla^{k}f+\left[\left(\frac{(n-1)(n-2)\rho-n}{(n-1)(n-2)}\right) R+\lambda-\frac{1}{n-1}\xi_{1}\right]R_{kj}\nabla^{k}f\\
			&+\left[\left(\frac{(1-(n-1)\rho)}{n-1}R-\lambda\right)\xi_{1}-\frac{1}{n-1}\left\langle\nabla\xi_{1},\nabla f\right\rangle-\frac{1}{n-2}\left(\vert Ric\vert^2-\frac{R^2}{n-1}\right)\right]\nabla_{j}f\\
			&+\frac{1}{n-1}\nabla_{k}f\nabla_{j}\xi_{1}\nabla^{k}f\\
			=&\frac{2}{n-2}\xi_{1}^2\nabla_{j}f+\left[\left(\frac{(n-1)(n-2)\rho-n}{(n-1)(n-2)}\right) R+\lambda-\frac{1}{n-1}\xi_{1}\right]\xi_{1}\nabla_{j}f\\
			&+\left[\left(\frac{(1-(n-1)\rho)}{n-1}R-\lambda\right)\xi_{1}-\frac{1}{n-1}\left\langle\nabla\xi_{1},\nabla f\right\rangle-\frac{1}{n-2}\left(\vert Ric\vert^2-\frac{R^2}{n-1}\right)\right]\nabla_{j}f\\
			&+\frac{1}{n-1}\nabla_{j}f\nabla_{k}\xi_{1}\nabla^{k}f\\
			=&\left[\frac{n}{(n-1)(n-2)}\xi_{1}^2+\left[\left(\frac{(n-1)(n-2)\rho-n}{(n-1)(n-2)}\right) R+\lambda\right]\xi_{1}\right]\nabla_{j}f\\
			&+\left[\left(\frac{(1-(n-1)\rho)}{n-1}R-\lambda\right)\xi_{1}-\frac{1}{n-1}\left\langle\nabla\xi_{1},\nabla f\right\rangle-\frac{1}{n-2}\left(\vert Ric\vert^2-\frac{R^2}{n-1}\right)\right]\nabla_{j}f\\
			&+\frac{1}{n-1}\left\langle\nabla\xi_{1},\nabla f\right\rangle\nabla_{j}f\\
			=&\left[\frac{n}{(n-1)(n-2)}\xi_{1}^2+\left[\left(\frac{(n-1)(n-2)\rho-n}{(n-1)(n-2)}\right) R+\lambda\right]\xi_{1}\right]\nabla_{j}f\\
			&+\left[\left(\frac{(1-(n-1)\rho)}{n-1}R-\lambda\right)\xi_{1}-\frac{1}{n-2}\left(\vert Ric\vert^2-\frac{R^2}{n-1}\right)\right]\nabla_{j}f\\
			=&\frac{1}{(n-1)(n-2)}\left[n\xi_{1}^2-2R\xi_{1}-\left((n-1)\vert Ric\vert^2-R^2\right)\right]\nabla_{j}f,
	\end{align*}
	where in the second line we have used that $df\wedge d\xi_{1}=0$ in the last term, which follows form \eqref{wedgezero2'}. Once $\mu$ is defined as in \eqref{eigbach}, we finally get the desired result.
\end{proof}

\begin{remark}
    In \cite[Proposition 4.1]{borges2}, the second author proved that $\nabla f$ is an eigenvector of the Bach tensor of a Schouten soliton, whose eigenvalue is given by \eqref{identschb=0}. Proposition \ref{Propgradeigv} extends this result for $\rho$-Einstein solitons, once when $\rho=1/2(n-1)$, we have $\xi_{1}=0$, and  \eqref{eigbach} becomes \eqref{identschb=0}.
\end{remark}

\section{Local structure of $\rho$-Einstein solitons with $B(\nabla f,\nabla f)\geq0$}\label{locstruc}

In this section, we deepen the understanding of the expression \eqref{eigbach}, which describes the eigenvalue corresponding to $\nabla f$. In the first result, Theorem \ref{proptwoeigen}, we show this eigenvalue has a sign, and when $B(\nabla f,\nabla f)$ is nonnegative, then we show in Proposition \ref{local_warp} that these $\rho$-Einstein solitons are locally warped products, which allows us to show that these solitons must be rotationally symmetric in low dimensions $n\in\{3,4\}$.

\begin{theorem}\label{proptwoeigen}
		Let $(M^n,g,f,\lambda)$ be a nontrivial gradient $\rho$-Einstein soliton with $n\geq3$. If $\rho= 0$, assume in addition that $dR\wedge df=0$. Then the eigenvalue corresponding to $\nabla f(p)\neq0$ satisfies
        \begin{align*}
            \mu=\left(\sum_{i=2}^{n}\xi_{i}\right)^2-(n-1)\sum_{i=2}^{n}\xi_{i}^2\leq0,
        \end{align*}
        where $\xi_{1},\xi_{2},\ldots,\xi_{n}$ are the eigenvalues of $Ric$. Furthermore, if $B(\nabla f,\nabla f)\geq0$ in $M$, then $Ric$ has at most two eigenvalues $\xi_{1}$ and $\xi_{2}$ in $\mathcal{R}$, which are smooth, constant on the regular levels of $f$, and have constant multiplicities $1$ and $n-1$, respectively, in this set.
\end{theorem}
\begin{proof}
    First of all, from \eqref{alleig}, we get
	\begin{align}\label{incoord}
		\vert Ric\vert^2=\sum_{i=1}^{n}\xi_{i}^2\ \ \ \text{and}\ \ \ R=\sum_{i=1}^{n}\xi_{i}.
	\end{align}
    Now, using \eqref{eigbach} we obtain
    \begin{align*}
		\mu&=n\xi_{1}^{2}-2R\xi_{1}-((n-1)\vert Ric\vert^2-R^2)\\
        &=n\xi_{1}^{2}-2\left(\sum_{i=1}^{n}\xi_{i}\right)\xi_{1}-(n-1)\sum_{i=1}^{n}\xi_{i}^2+\left(\sum_{i=1}^{n}\xi_{i}\right)^2\\
        &=(n-2-n+1)\xi_{1}^{2}-2\left(\sum_{i=2}^{n}\xi_{i}\right)\xi_{1}-(n-1)\sum_{i=2}^{n}\xi_{i}^2+\left(\sum_{i=2}^{n}\xi_{i}\right)^2+2\left(\sum_{i=2}^{n}\xi_{i}\right)\xi_{1}+\xi_{1}^{2}\\
        &=-\xi_{1}^{2}-(n-1)\sum_{i=2}^{n}\xi_{i}^2+\left(\sum_{i=2}^{n}\xi_{i}\right)^2+\xi_{1}^{2}\\
        &=\left(\sum_{i=2}^{n}\xi_{i}\right)^2-(n-1)\sum_{i=2}^{n}\xi_{i}^2.
	\end{align*}
    On the other hand, using the Cauchy-Schwarz inequality, we get
    \begin{align*}
		\mu=\left(\sum_{i=2}^{n}\xi_{i}\right)^2-(n-1)\sum_{i=2}^{n}\xi_{i}^2\leq(n-1)\sum_{i=2}^{n}\xi_{i}^2-(n-1)\sum_{i=2}^{n}\xi_{i}^2=0.
	\end{align*}
    As a consequence, if $B(\nabla f,\nabla f)\geq0$, then $\mu=0$. The equality in Cauchy-Schwarz inequality now guarantees that $\xi_{2}=\cdots=\xi_{n}$.
	
	Now we will show that $\xi_{1}$ and $\xi_{2}$ are smooth functions on the set of regular points of $f$. For $\xi_{1}$ this follows immediately from \eqref{xi1}, and for $\xi_{2}$ this follows from the identity
	\begin{align*}
		R=\xi_{1}+(n-1)\xi_{2}.
	\end{align*}
	
	To see that $\xi_{1}$ and $\xi_{2}$ are constant on the regular levels of $f$, consider $V\in\mathfrak{X}(M)$ so that $V\perp\nabla f$. First, notice that \eqref{wedgezero2'} implies $V(\xi_{1})=0$. On the other hand, as $V(R)=0$, we get $(n-1)V(\xi_{2})=V(R)-V(\xi_{1})=0$. This shows that $\xi_{1}$ and $\xi_{2}$ are constant on the regular levels of $f$.
\end{proof}

It is worth pointing out the following consequence of Theorem \ref{proptwoeigen}, which has independent interest.
\begin{corollary}\label{signRScase}
    Let $(M^n,g,f,\lambda)$ be a nontrivial gradient Ricci soliton which is rectifiable. Then
     \begin{align*}
      B(\nabla f,\nabla f)=\frac{\mu}{(n-1)(n-2)^2}\vert\nabla f\vert^2\leq0.
  \end{align*}
\end{corollary}

\begin{remark}
    Rectifiable Ricci solitons were considered in \cite{petersen2}. Among these solitons, are those with constant scalar curvature, studied in \cite{ferlo}. It should be pointed out that there are numerous examples of rectifiable Ricci solitons with nonconstant scalar curvature, built on cohomogeneity one manifolds. These examples can be found in \cite{bdw,bdw2,dw,dw2} and in references therein.
\end{remark}

A consequence of Theorem \ref{proptwoeigen} is the following result, which decomposes the metric locally as a warped product between an interval and an Einstein manifold.

\begin{proposition}\label{local_warp}
	Let $(M^n,g,f,\lambda)$ be a nontrivial gradient $\rho$-Einstein soliton with $n\geq3$ and $B(\nabla f,\nabla f)\geq0$ on $M$. If $\rho=0$, assume in addition that $dR\wedge df=0$. Given a regular point $p\in M$, there are an open set $U\subset M$, an Einstein manifold $(F^{n-1},g_{_{F}})$ and a positive smooth function $h:I\rightarrow\mathbb{R}$, so that
	\begin{align}\label{locdecomp}
		U=I\times F^{n-1}\ \ \ \text{and}\ \ \ g\vert_{_{U}}=dt^2+h^2g_{_{F}},
	\end{align}
	where $I$ is an interval.
\end{proposition}
\begin{proof}
	Set $F^{n-1}$ as the connected component of $f^{-1}(f(p))$ containing $p$, and consider in it the induced metric, which we denote by $g_{_{F}}$. Let $U\subset M$ be diffeomorphic to $I\times F^{n-1}$ with $p\in U$, where $I$ is the connected integral curve of $\frac{\nabla f}{\vert\nabla f\vert}$ through $p$. Notice that $\vert\nabla f(q)\vert^2\neq0$, $\forall q\in U$. We will simply write $U=I\times F^{n-1}$.
	
	Let $(s,x_{2},\ldots,x_{n})\in U$, where $(x_{2},\ldots,x_{n})$ are coordinates of $F^{n-1}$ and $s$ is the arc length parameter of $I$. Notice that $F^{n-1}=\{(t,x_{2},\ldots,x_{n})\in U\vert s=0\}$, $\partial_{t}=\frac{\nabla f}{\vert\nabla f\vert}$ and $g(\partial_{t},\partial_{a})=0,\ \forall a\in\{2,\ldots,n\}$. Furthermore, Theorem \ref{proptwoeigen} asserts that $Ric(\partial_{a},X)=\xi_{2}g(\partial_{a},X),\ \forall X\in\mathfrak{X}(M)$ and $\forall a\in\{2,\ldots,n\}$.
	
	Placing $\partial_{a}$ and $\partial_{b}$ in \eqref{fundeq}, and using $\nabla f=\vert\nabla f\vert\partial_{t}$, we obtain
	\begin{align*}
		\vert\nabla f\vert g(\nabla_{\partial_{a}}\partial_{t},\partial_{b})&=g(\nabla_{\partial_{a}}(\vert\nabla f\vert\partial_{t}),\partial_{b})\\
		&=g(\nabla_{\partial_{a}}\nabla f,\partial_{b})\\
		&=\left(\rho R+\lambda-\xi_{2}\right)g_{ab},
	\end{align*}
	and thus,
	\begin{align}\label{toint}
		\begin{split}
			\partial_{t}g_{ab}&=g(\nabla_{\partial_{t}}\partial_{a},\partial_{b})+g(\partial_{a},\nabla_{\partial_{t}}\partial_{b})\\
		&=g(\nabla_{\partial_{a}}\partial_{t},\partial_{b})+g(\partial_{a},\nabla_{\partial_{b}}\partial_{t})\\
		&=2\left(\frac{\rho R+\lambda-\xi_{2}}{\vert\nabla f\vert}\right)g_{ab}.
		\end{split}
	\end{align}
	Now, Theorem \ref{proptwoeigen} assures that $R$, $\xi_{2}$, $f$ and
	\begin{align*}
		u=\frac{\rho R+\lambda-\xi_{2}}{\vert\nabla f\vert}
	\end{align*}
	are functions depending only on $t$ within $U$. Integrating $\partial_{t}g_{ab}=2u(t)g_{ab}$, we get
	\begin{align}
			g_{ab}(t,x_{2},\ldots,x_{n})=(h(t))^{2}(g_{_{F}})_{ab}(x_{2},\ldots,x_{n}),
	\end{align}
	where $h:I\rightarrow\mathbb{R}$ is given by $h(t)=e^{\int_{0}^{t}u(y)dy}$ and
	\begin{align*}
		(g_{_{F}})_{ab}(x_{2},\ldots,x_{n})=g_{ab}(0,x_{2},\ldots,x_{n})
	\end{align*}
	is the metric of $F^{n-1}$, induced by the metric of $M$.
	
	Using \eqref{fundeq} and the equations for the Ricci tensor and covariant derivative of a warped product, which can be found in Chapter 7 of \cite{oneill}, we get
	\begin{align}\label{rcexpwp}
		(Ric_{F})_{ab}=(hh''+(n-2)(h')^2-hh'f'+(\rho R+\lambda)h^2)(g_{F})_{ab},
	\end{align}
	for all $a,b\in\{2,\ldots,n\}$, where $Ric_{F}$ is the Ricci curvature of $F^{n-1}$, and the prime $(\cdot)'$ is the derivative with respect to $t$. Once the function on the right hand side of \eqref{rcexpwp} multiplying $g_{F}$ depends only on $t$, we conclude that $(F^{n-1},g_{F})$ is Einstein.	
\end{proof}

\begin{proof}[{\bf Proof of Theorem \ref{main_thm}}]
    Assume $B(\nabla f,\nabla f)\geq0$ and let $p\in\mathcal{R}$. By Proposition \ref{local_warp}, there are an open set $U\subset M$ containing $p$, an Einstein manifold $(F^{n-1},g_{_{F}})$ and a positive smooth function $h:I\rightarrow\mathbb{R}$, so that \eqref{locdecomp} holds true, where $I$ is an interval. The proof is complete if $\rho\neq0$. If $\rho=0$, it follows from Remark \ref{remarkRectf_rho0}, we see that $df\wedge dR=0$ is equivalent to rectifiability, finishing the proof.
\end{proof}

\begin{corollary}\label{wey_har}
     Let $(M^n,g,f,\lambda)$ be a nontrivial gradient $\rho$-Einstein soliton with $n\geq3$ and $B(\nabla f,\nabla f)\geq0$ on $M$. If $\rho= 0$, assume in addition that $dR\wedge df=0$. Then $M$ is locally conformally flat if $n\in\{3,4\}$ and is Weyl harmonic if $n\geq5$.
\end{corollary}
\begin{proof}
    It follows from \eqref{locdecomp} that at our conditions, the $\rho$-Einstein soliton is locally a warped product, with Einstein fiber $F^{n-1}$. Notice that as $F^{n-1}$ is Einstein, it follows from Example 16.26 (i) of Chapter 16 in \cite{besse} that $I\times_{h}F^{n-1}$ has harmonic Weyl tensor. Now, if $n\in\{3,4\}$, then $n-1\in\{2,3\}$, which makes $F^{n-1}$ to be of constant sectional curvature, implying $I\times_{h}F^{n-1}$ to be locally conformally flat \cite{brovaz}, proving the corollary.
\end{proof}

In the next proposition we will use \eqref{wey_har} to show that $B$ vanishes identically if $B(\nabla f,\nabla f)$ is nonnegative on $M$. We first highlight the following lemma, which is a consequence of previous computations.

\begin{lemma}
    Let $(M^n,g,f,\lambda)$ be a nontrivial gradient $\rho$-Einstein soliton, and assume that $B(\nabla f,\nabla f)\geq0$ on $M$. If $\rho= 0$, assume in addition that $dR\wedge df=0$. Then, for any $X,Y,Z\in\mathfrak{X}(M)$:
    \begin{align}
        &(divRm)(X,\nabla f,Y)=\frac{1}{2(n-1)}(\left\langle \nabla f,\nabla R\right\rangle\left\langle X,Y\right\rangle-\left\langle \nabla f,X\right\rangle\left\langle \nabla R,Y\right\rangle),\label{eq_1}\\
        &Rm(X,Y,Z,\nabla f)=\frac{1-2(n-1)\rho}{2(n-1)}(\left\langle \nabla R,Y\right\rangle\left\langle X,Z\right\rangle-\left\langle \nabla R,X\right\rangle\left\langle Y,Z\right\rangle).\label{eq_2}
    \end{align}
\end{lemma}
\begin{proof}
From the vanishment of the Cotton tensor, \eqref{divricc} and \eqref{cotton} imply
\begin{align*}
    (divRm)_{jik}=\nabla_{i}R_{jk}-\nabla_{j}R_{ik}=\frac{1}{2(n-1)}(g_{jk}\nabla_{i}R-g_{ik}\nabla_{j}R),
\end{align*}
which proves \eqref{eq_1}. Finally, using \eqref{cotton_rho} we get right away that
\begin{align*}
    R_{jikl}\nabla_{l}f=\left(\frac{1-2(n-1)\rho}{2(n-1)}\right)(\nabla_{i}Rg_{jk}-\nabla_{j}Rg_{ik}),
\end{align*}
proving \eqref{eq_2}.
\end{proof}

Notice that \eqref{eq_1} allows us to rewrite Lemma \ref{lemma1} eliminating the term $(divRm)_{ijk}\nabla_{j}f$, and making the expression useful to prove the next result.

\begin{proposition}\label{prop1} Let $(M^n,g,f,\lambda)$ be a nontrivial gradient $\rho$-Einstein soliton, and assume $B(\nabla f,\nabla f)\geq0$. If $\rho= 0$, assume in addition that $dR\wedge df=0$. Then
\begin{align}
    &B(X,Y)=0,\ \forall X,Y\in \mathfrak{X}(M)\label{part1}.
\end{align}
\end{proposition}
\begin{proof}
Let $\{e_{1},\ldots,e_{n}\}$ be an orthonormal basis diagonalizing $Ric$ at $p\in\mathcal{R}$, with $e_{1}=\frac{\nabla f(p)}{\vert\nabla f(p)\vert}$. We notice that $B(e_{1},e_{i})=0$, $\forall i\in\{1,\ldots,n\}$, follow from \eqref{gradeigv}. In what follows we are going to show that $B(e_{a},e_{b})=0$, $\forall a,b\in\{2,\ldots,n\}$. First, we consider the expression for $B$ given in Lemma \ref{lemma1}, and use equations \eqref{firstdereigID} and \eqref{eq_1} to reduce the computation of $B(e_{a},e_{b})$ to an algebraic level. In order to do so, notice that equation \eqref{eq_1} is equivalent to
\begin{align}\label{eq_1'}
    (divRm)(\cdot,\nabla f,\cdot)=\frac{1}{2(n-1)}(\left\langle \nabla f,\nabla R\right\rangle g-df\otimes dR).
\end{align}

Inserting equations \eqref{firstdereigID} and \eqref{eq_1'} in \eqref{firstbachstep}, we conclude that
\begin{eqnarray*}
(n-2)B&=&\frac{2}{n-2}Ric^2+\left[\lambda-\left(\frac{n-(n-1)(n-2)\rho}{(n-1)(n-2)}\right) R-\frac{1}{n-1}\xi_{1}\right]Ric\\
&-&\left[\frac{1}{n-2}\left(\vert Ric\vert^2-\frac{R^2}{n-1}\right)+\left(\frac{(1-2(n-1)\rho)}{2(n-1)}\right)\Delta R\right]g\\
&+&\frac{1}{n-1}d\xi_{1}\otimes df-\frac{1}{2(n-1)}df\otimes dR+\frac{1}{n-1}\xi_{1}(\rho R+\lambda)g+\frac{1}{2(n-1)}\left\langle\nabla f,\nabla R\right\rangle g
\end{eqnarray*}
Now, we collect all terms multiplying $g$ in the expression above denoted by $T$ and use \eqref{IdentitySch} to get

\begin{align*}
&T=-\frac{1}{n-2}\left(\vert Ric\vert^2-\frac{R^2}{n-1}\right)-\left(\frac{(1-2(n-1)\rho)}{2(n-1)}\right)\Delta R+\frac{1}{n-1}\xi_{1}(\rho R+\lambda)+\frac{1}{2(n-1)}\left\langle\nabla f,\nabla R\right\rangle\\
&=\frac{1}{n-2}\left(\frac{R^2}{n-1}-\vert Ric\vert^2\right)+\frac{1}{2(n-1)}\left(2|Ric|^2-2R\left(\rho R+\lambda\right)+2\xi_{1}(\rho R+\lambda)\right)\\
&=\frac{1}{n-2}\left(\frac{R^2}{n-1}-\vert Ric\vert^2\right)+\frac{1}{n-1}\left(|Ric|^2-R\left(\rho R+\lambda\right)\right)+\frac{1}{n-1}\xi_{1}(\rho R+\lambda)\\
&=\frac{1}{n-1}\left(\frac{1}{n-2}(R^2-\vert Ric\vert^2)-R\left(\rho R+\lambda\right)+\xi_{1}(\rho R+\lambda)\right).
\end{align*}

%\begin{eqnarray*}
% T&=&-\frac{1}{n-2}\left(\vert Ric\vert^2-\frac{R^2}{n-1}\right)-\left(\frac{(1-2(n-1)\rho)}{2(n-1)}\right)\Delta R+\frac{1}{n-1}\xi_{1}(\rho R+\lambda)+\frac{1}{2(n-1)}\left\langle\nabla f,\nabla R\right\rangle\\
%&=&\frac{1}{n-2}\left(\frac{R^2}{n-1}-\vert Ric\vert^2\right)+\frac{1}{2(n-1)}\left(2|Ric|^2-2R\left(\rho R+\lambda\right)-\left\langle\nabla f,\nabla R\right\rangle+\left\langle\nabla f,\nabla R\right\rangle+2\xi_{1}(\rho R+\lambda)\right)\\
%&=&\frac{1}{n-2}\left(\frac{R^2}{n-1}-\vert Ric\vert^2\right)+\frac{1}{n-1}\left(|Ric|^2-R\left(\rho R+\lambda\right)\right)+\frac{1}{n-1}\xi_{1}(\rho R+\lambda)\\
%&=&\frac{1}{n-1}\left(\frac{1}{n-2}(R^2-\vert Ric\vert^2)-R\left(\rho R+\lambda\right)+\xi_{1}(\rho R+\lambda)\right).
%\end{eqnarray*}

Returning these computations to the expression of $B$ found above, we obtain
\begin{eqnarray*}
(n-2)B&=&\frac{2}{n-2}Ric^2+\left[\lambda-\left(\frac{n-(n-1)(n-2)\rho}{(n-1)(n-2)}\right) R-\frac{1}{n-1}\xi_{1}\right]Ric\\
&+&\frac{1}{n-1}\left[\frac{1}{n-2}(R^2-\vert Ric\vert^2)-R\left(\rho R+\lambda\right)+\xi_{1}(\rho R+\lambda)\right]g\\
&\textcolor{red}{-}&\frac{1}{2(n-1)}df\otimes d(R+2\xi_{1}).
\end{eqnarray*}

Recall that $df(e_{a})=0$ and, in view of $\xi_{2}=\xi_{3}=\cdots=\xi_{n}$, we may write $Ric(e_{a},e_{b})=\xi_{2}\delta_{ab}$. Thus,
\begin{eqnarray*}
    (n-2)B_{ab}&=&\frac{2}{n-2}\xi_{2}^2\delta_{ab}+\left[\lambda-\left(\frac{n-(n-1)(n-2)\rho}{(n-1)(n-2)}\right) R-\frac{1}{n-1}\xi_{1}\right]\xi_{2}\delta_{ab}\\
    &+&\frac{1}{n-1}\left[\frac{1}{n-2}(R^2-\vert Ric\vert^2)-R\left(\rho R+\lambda\right)+\xi_{1}(\rho R+\lambda)\right]\delta_{ab},
\end{eqnarray*}
and consequently, $B_{ab}=0$ if $a\neq b$. On the other hand, setting $a=b$ and using $R=\xi_{1}+(n-1)\xi_{2}$ and $\vert Ric\vert^2=\xi_{1}^2+(n-1)\xi_{2}^2$, we have
\begin{eqnarray*}
\kappa B_{aa}&=&2(n-1)\xi_{2}^2+\left[(n-1)(n-2)\lambda-\left(n-(n-1)(n-2)\rho\right) R-(n-2)\xi_{1}\right]\xi_{2}\\
&+&R^2-\vert Ric\vert^2-(n-2)R\left(\rho R+\lambda\right)+(n-2)\xi_{1}(\rho R+\lambda)\\
 &=&2(n-1)\xi_{2}^2+(n-1)(n-2)\lambda\xi_{2}-\left(n-(n-1)(n-2)\rho\right) R\xi_{2}-(n-2)\xi_{1}\xi_{2}\\
&+&(1-(n-2)\rho)R^2-\vert Ric\vert^2-(n-2)\lambda R+(n-2)\xi_{1}(\rho R+\lambda)\\
&=&2(n-1)\xi_{2}^2-\left(n-(n-1)(n-2)\rho\right)(\xi_{1}+(n-1)\xi_{2})\xi_{2}-(n-2)\xi_{1}\xi_{2}\\
&+&(1-(n-2)\rho)(\xi_{1}+(n-1)\xi_{2})^2-(\xi_{1}^2+(n-1)\xi_{2}^2)+(n-1)(n-2)\lambda\xi_{2}\\
&-&(n-2)\lambda (\xi_{1}+(n-1)\xi_{2})+(n-2)\xi_{1}(\rho (\xi_{1}+(n-1)\xi_{2})+\lambda)\\
&=&(1-(n-2)\rho-1+(n-2)\rho)\xi^{2}_{1}\\
&+&(2(n-1)-(n-(n-1)(n-2)\rho)(n-1)+(n-1)^2(1-(n-2)\rho)-n+1)\xi^{2}_{2}\\
&+&(-n+(n-1)(n-2)\rho-n+2+2(n-1)(1-(n-2)\rho)+(n-2)(n-1)\rho)\xi_{1}\xi_{2}\\
&+&(-(n-2)\lambda+(n-2)\lambda)\xi_{1}+(-(n-2)(n-1)\lambda+(n-1)(n-2)\lambda)\xi_{2}\\
&=&(n-1-n(n-1)+(n-1)^2(n-2)\rho+(n-1)^2-(n-1)^2(n-2)\rho)\xi^{2}_{2}\\
&+&(-n+(n-1)(n-2)\rho-n+2+2(n-1)-2(n-1)(n-2)\rho+(n-2)(n-1)\rho)\xi_{1}\xi_{2}\\
&=&0,
\end{eqnarray*}
where $\kappa=(n-1)(n-2)^2$. This shows that $B_{aa}=0$, finishing the proof.
\end{proof}

\begin{proof}[{\bf Proof of Corollary \ref{main_cor}}]
    The proof follows combining Corollary \ref{wey_har} and Proposition \ref{prop1}.
\end{proof}

\begin{remark}
     Putting together Proposition \ref{prop1}, the classification of Bach flat Ricci solitons and its analyticity, we conclude by Corollary \ref{signRScase} that for all rectifiable Ricci solitons that are not rotationally symmetric, we should have $B(\nabla f,\nabla f)<0$ in a dense set of $M$. In particular, this happens for those built on \cite{bdw,bdw2,dw,dw2}.
\end{remark}

\section{Locally conformally flat $\rho$-Einstein solitons}\label{locaconf}

In this section we show that complete locally conformally flat $\rho$-Einstein solitons with nonnegative sectional curvature are rotationally symmetric. In the sequence, we classify those for which $\lambda=0$.

\begin{theorem}\label{theorem_rotationally}
	Let $(M^n,g,f,\lambda)$, $n\geq3$, be a complete simply connected nontrivial gradient $\rho$-Einstein soliton with nonnegative sectional curvature and locally conformally flat. Then there are $\Lambda\subset M$ with at most two points and a smooth function $h:(a,b)\rightarrow(0,\infty)$ such that $M\backslash\Lambda$ is isometric to $(a,b)\times_{h}\mathbb{S}^{n-1}$, with $-\infty\leq a<b\leq\infty$. In other words, $M^n$ is rotationally symmetric.
\end{theorem}

\begin{proof}
    Once $M$ is locally conformally flat, the Bach tensor of $M$ vanishes identically. From Theorem \ref{main_thm}, we know that around a regular point $p\in M$ of $f$ there is an open set $U\subset M$ isometric to a warped product $(a,b)\times_{h}F^{n-1}$, where $F^{n-1}$ is an Einstein manifold. Once $M$ is locally conformally flat, Theorem 1 of \cite{brovaz0} (or Theorem 2 of \cite{brovaz}) imply that $F^{n-1}$ is a space form of constant sectional curvature $K_{F}$.
    
    In what follows we show that there is a set $\Lambda\subset M$ so that $M\backslash\Lambda$ is isometric to $(a,b)\times_{h}\mathbb{S}^{n-1}$, with $-\infty\leq a<b\leq\infty$. %We follow the ideas established in proofs of Theorem 1 and 2 of \cite{fer_gar}.

    First notice that once the open set $U\subset M$ is isometric to the warped product $(a,b)\times_{h}F^{n-1}$, $F^{n-1}$ is a totally umbilical hypersurface of $M$ \cite{oneill}, that is, its second fundamental form satisfies $II=\psi g_{F}$, for some function $\psi$ on $F$. Consequently, the Gauss formula and $sec_{M}\geq0$ implies
    \begin{align}\label{gauss_eq}
        \begin{split}
            sec_{F}(X,Y)&=sec_{M}(X,Y)+II(X,X)II(Y,Y)-(II(X,Y))^2\\
            &=sec_{M}(X,Y)+\psi^2(g_{F}(X,X)g_{F}(Y,Y)-(g_{F}(X,Y))^2)\\
            &\geq0,
        \end{split}
    \end{align}
    for all sets $\{X,Y\}$ of L.I. vectors on $T_{q}F$, and all $q\in F$. Since $sec_{F}$ is a constant, we only need to consider two cases, which after normalization are $K_{F}=0$ or $K_{F}=1$.

    Suppose $K_{F}=0$. In this case, \eqref{gauss_eq} implies that $\psi$ vanishes identically. Consequently, $F^{n-1}$ is totally geodesic, which implies that $U$ is isometric to the product manifold $(a,b)\times F^{n-1}$. In particular, $sec_{M}=0$ on $U$. Once by Corollary \ref{densedense} the set of regular points of $f$ is dense in $M$, we conclude that $M^{n}$ is flat, and hence isometric to the Gaussian soliton $\mathbb{R}^{n}$, which is rotationally symmetric.

    Suppose $K_{F}=1$. In this case, $U$ is isometric to $(a,b)\times_{h}F^{n-1}$, where $F^{n-1}$ is an open set of $\mathbb{S}^{n-1}$. From now on, we argue as in the proofs of Theorem 1 and 2 of \cite{fer_gar}. Once $M$ is simply connected, $F^{n-1}$ is an embedded two-sided hypersurface of $M$. Consider the signed distance from $F^{n-1}$, $d_{F}:M\rightarrow\mathbb{R}$, defined as $d_{F}(p)=\inf_{x\in F}\{d(p,x)\}$. We are going to consider $F^{n-1}=f^{-1}(0)$, which means at our conditions that the hypersurface parallel to $F^{n-1}$ at distance $t\in\mathbb{R}$ is given by $f^{-1}(t)$. Let us first consider the positive side of $M$. Consider the nonempty set $I_{+}=\{b\in\mathbb{R}\vert g=dt^2+h^{2}g_{F}\ \text{in}\ d_{F}^{-1}([0,b])\}$. Following the same argument of \cite{fer_gar}, based on the fact that the set of regular points of $f$ is dense in $M$, we see that $[0,b]\subset I_{+}$, whenever $h(t)>0$, $\forall t\in[0,b]$. Thus, either $I_{+}=[0,\infty)$ or $I_{+}=[0,b)$, for some $b\in(0,\infty)$. For the negative side, the same arguments show that the set $I_{-}=\{a\in\mathbb{R}\vert g=dt^2+h^{2}g_{F}\ \text{in}\ d_{F}^{-1}([a,0])\}$ is either equals to $(-\infty,0]$ or to $(a,0]$, for some $a\in(-\infty,0)$. This finishes the proof.
\end{proof}

\begin{proof}[{\bf Proof of Theorem \ref{thm_classif}}]
        First, arguing as in the first part of the proof of Theorem \ref{theorem_rotationally}, once $M$ is locally conformally flat, $M$ is locally isometric to a warped product $(a,b)\times_{h}F^{n-1}$, where $F^{n-1}$ is a space form of constant sectional curvature $K_{F}$. Now, proceeding as in the last part of Theorem \ref{theorem_rotationally}, there is a set $\Lambda\subset M$ with at most two points so that $M\backslash\Lambda$ is isometric to $(a,b)\times_{h}F^{n-1}$, with $-\infty\leq a<b\leq\infty$, where $F^{n-1}$ is either $\mathbb{S}^{n-1}$, $\mathbb{R}^{n-1}$ or $\mathbb{H}^{n-1}$. Now we use Theorems \ref{rotat_=0} and \ref{rotat_=04}, proved in \cite{catino}, to conclude the proof.
\end{proof}

\begin{proof}[{\bf Proof of Corollary \ref{cor_W=0}}]
    From item \eqref{item3} of Corollary \ref{main_cor} and the assumption $n\in\{3,4\}$, it follows that $M$ is locally conformally flat. The result now follows from Theorem \ref{thm_classif}.
\end{proof}

Following the same steps in the proof of Theorem \ref{thm_classif}, and Theorem \ref{rotat_=04}, we obtain the following corollary: 

\begin{corollary}
	Let $(M^n,g,f,\lambda)$, $n\geq3$, be a complete simply connected and locally conformally flat nontrivial gradient $\rho$-Einstein soliton, with nonnegative sectional curvature and $Ric(\nabla f,\nabla f)$ vanishing identically. Then $M$ is isometric to $\mathbb{R}\times\mathbb{H}^{n-1}$, $\mathbb{R}^{n}$ or $\mathbb{R}\times\mathbb{S}^{n-1}$.
\end{corollary}

\section{Declarations}

\noindent {\bf Author contributions:} The authors contributed equally to this work. 

\vspace{.5cm}

\noindent {\bf Data availability:} Data sharing is not applicable to this article as no datasets were generated or analyzed during the current study.

\vspace{.5cm}

\noindent {\bf Conflict of interest:} The authors declare that there is no conflict of interest.

\vspace{.5cm}

\noindent {\bf Funding:} The first author was supported by Brazilian National Council for Scientific and Technological Development and by FAPITEC/SE/Brazil. The second and third author did not receive any funding for conducting this study. 

\vspace{.5cm}

\noindent {\bf Acknowledgments:} The first author has been partially supported by Brazilian National Council for Scientific and Technological Development (CNPq Grants 408834/2023-4 and
403869/2024-2) and FAPITEC/SE/Brazil 019203.01303/2024-1.

\vspace{.5cm}

\bibliographystyle{amsplain}

\begin{thebibliography}{10}
   \bibitem{agila} Agila E. F. and Gomes, J. N. \textit{Geometric and analytic results for Einstein solitons}. Mathematische Nachrichten. 2024.

    \bibitem{bach} Bach, R. \textit{Zur weylschen relativitätstheorie und der weylschen erweiterung des krümmungstensorbegriffs}. Mathematische Zeitschrift, 9(1), 110-135, 1921.

    \bibitem{berg} Bergman, J. \textit{Conformal Einstein spaces and Bach tensor generalization in n dimensions}. Linkopings Universitet (Sweden); 2004.
    
    \bibitem{besse} Besse A. \textit{Einstein Manifolds}. Springer Science and Business Media. 2007.
    
	\bibitem{borges} Borges, V. \textit{On complete gradient Schouten solitons}. Nonlinear Analysis 221, p.112883, 2022.

    \bibitem{borges2} Borges, V. \textit{Rigidity of Bach-flat gradient Schouten solitons}. manuscripta mathematica, pp.1-11, 2024.

    \bibitem{borges_gomes} Borges, V. and Gomes, J. N. \textit{Gradient estimates for $\rho$-Einstein solitons and applications}. Preprint, 2025.

    \bibitem{bryant} Bryant, R. \textit{Ricci flow solitons in dimension three with $SO(3)$-symmetries}. Preprint, 2005.

    \bibitem{brovaz0} Brozos-Vázquez, M., Garcıa-Rıo, E. and Vázquez-Lorenzo R. \textit{Some remarks on locally conformally flat static space–times}. Journal of mathematical physics, 46(2). 2005.

    \bibitem{brovaz} Brozos-Vázquez, M., Garcıa-Rıo, E. and Vázquez-Lorenzo R. \textit{Warped product metrics and locally conformally flat structures.} Matemática Contemporânea. 28(5):91-110, 2005.

    \bibitem{bdw} Buzano, M., Dancer, A., Gallaugher, M. and Wang, M. \textit{Non-Kähler expanding Ricci solitons, Einstein metrics, and exotic cone structures}. Pacific Journal of Mathematics, 273(2), pp.369-394, 2014.

    \bibitem{bdw2} Buzano, M., Dancer, A. and Wang, M. \textit{A family of steady Ricci solitons and Ricci-flat metrics}. Communications in Analysis and Geometry 23, no. 3, 2015.

    \bibitem{caoqian_0} Cao, H. D. and Chen, Q. \textit{On locally conformally flat gradient steady Ricci solitons}. Transactions of the American Mathematical Society, 364(5), 2377-2391, 2012.
    
	\bibitem{caoqian} Cao, H. D. and Chen, Q. \textit{On Bach-flat gradient shrinking Ricci solitons}. Duke Mathematical Journal 162, no. 6, 1149-1169, 2013.
	
	\bibitem{caocatino} Cao, H.D., Catino, G., Chen, Q., Mantegazza, C. and Mazzieri, L. \textit{Bach-flat gradient steady Ricci solitons}. Calculus of Variations and Partial Differential Equations 49, no. 1, 125-138, 2014.
	
	\bibitem{catino} Catino, G., Mazzieri, L. \textit{Gradient Einstein solitons}. Nonlinear Analysis 132, 66-94, 2016.
	
	\bibitem{catino1} Catino, G., Mazzieri, L., Mongodi, S. \textit{Rigidity of gradient Einstein shrinkers}. Communications in Contemporary Mathematics 17.06, 1550046, 2015.

	\bibitem{catino2} Catino, G., Cremaschi, L., Djadli, Z., Mantegazza, C., Mazzieri, L. \textit{The Ricci–Bourguignon flow}. Pacific Journal of Mathematics, v. 287, n. 2, p. 337-370, 2017.	

    \bibitem{choi} Choi, H. I. \textit{Characterizations of simply connected rotationally symmetric manifolds}. Transactions of the American Mathematical Society, 275(2), pp.723-727, 1983.
    
	\bibitem{chowluni} Chow, B., Lu, P., Ni, L. \textit{Hamilton’s Ricci flow}. American Mathematical Society, Science Press, 2023.

    \bibitem{dw} Dancer, A. and Wang, M.Y. \textit{Some New Examples of Non-Kähler Ricci Solitons}. Mathematical research letters, 16(2), pp.349-363, 2009.

    \bibitem{dw2} Dancer, A. and Wang, M.Y. \textit{On Ricci solitons of cohomogeneity one}. Annals of Global Analysis and Geometry 39, 259-292, 2011.
    
    \bibitem{derd} Derdziński, A. \textit{Self-dual Kähler manifolds and Einstein manifolds of dimension four}. Compositio Mathematica, 49(3), 405-433, 1983.
    
	\bibitem{pointofview} Eminenti, M., La Nave, G., Mantegazza, C. \textit{Ricci solitons: the equation point of view}. manuscripta mathematica, 127(3), 345-367, 2008.

    \bibitem{ferlo_0} Fernández-López, M., García-Río, E. \textit{Rigidity of shrinking Ricci solitons}. Mathematische Zeitschrift, 269:461-6, 2011.

    \bibitem{fer_gar} Fernández-López, M., García-Río, E. \textit{A note on locally conformally flat gradient Ricci solitons}. Geometriae Dedicata, 168, pp.1-7, 2014.
    
	\bibitem{ferlo} Fernández-López, M., García-Río, E. \textit{On gradient Ricci solitons with constant scalar curvature}. Proceedings of the American Mathematical Society, 144(1), pp.369-378, 2016

    \bibitem{fielder} Fiedler, B., Schimming, R. \textit{Exact solutions of the Bach field equations of general relativity}. Reports on Mathematical Physics. Feb 1;17(1):15-36, 1980.
    
	\bibitem{hamilton0} Hamilton, R. S. \textit{Three-manifolds with positive Ricci curvature}. Journal of Differential Geometry, v. 17, n. 2, p. 255-306, 1982.

	\bibitem{hamilton1} Hamilton, R. S. \textit{The Ricci Flow on Surfaces}. Contemporary Mathematics, v. 71, p. 237-361, 1988.

    \bibitem{hamilton2} Hamilton, R. S. \textit{The formations of singularities in the Ricci Flow}. Surveys in differential geometry, Vol. II (Cambridge, MA, 1993), 7--136, Int. Press, Cambridge, MA, 1995.

    \bibitem{ivey} Ivey, T. A. \textit{Local existence of Ricci solitons}. manuscripta mathematica, 91(2), pp.151-162, 1996.
    
	\bibitem{patuho} Ho, P.T. \textit{On the Ricci–Bourguignon flow}. International Journal of Mathematics, 31(06), p.2050044, 2020.

    \bibitem{muse} Munteanu, O., Sesum, N. {\it On gradient Ricci solitons}. Journal of Geometric Analysis, 23, pp.539-561, 2013.
    
	\bibitem{mi} Mi, R. \textit{Remarks on scalar curvature of gradient Ricci-Bourguignon sollitons}. Bulletin des Sciences Mathématiques, 171, 103034, 2021.
	
	\bibitem{oneill} O'Neill, B. {\it Semi-Riemannian geometry with applications to relativity.} Academic Press, 1983.
	
	\bibitem{petersen} Petersen, P., Wylie, W. \textit{Rigidity of gradient Ricci solitons}. Pacific journal of mathematics, v. 241, n. 2, p. 329-345, 2009.

    \bibitem{petersen2} Petersen P, Wylie W. \textit{On gradient Ricci solitons with symmetry}. Proceedings of the American Mathematical Society. 137(6):2085-92, 2009.

    \bibitem{swann} Swann A, Pedersen H. \textit{Einstein-Weyl geometry, the Bach tensor and conformal scalar curvature}. Journal für die reine und angewandte Mathematik, 441, pp. 99-113, 1993.

\end{thebibliography}

\end{document}